\documentclass[12pt,reqno]{article}

\usepackage[usenames]{color}
\usepackage{amssymb}
\usepackage{amsmath}
\usepackage{amsthm}
\usepackage{amsfonts}
\usepackage{amscd}
\usepackage{graphicx}

\usepackage{pstricks}

\usepackage[colorlinks=true,
linkcolor=webgreen,
filecolor=webbrown,
citecolor=webgreen]{hyperref}

\definecolor{webgreen}{rgb}{0,.5,0}
\definecolor{webbrown}{rgb}{.6,0,0}

\usepackage{color}
\usepackage{fullpage}
\usepackage{float}

\usepackage{graphics}
\usepackage{latexsym}
\usepackage{epsf}
\usepackage{breakurl}

\newcommand{\seqnum}[1]{\href{https://oeis.org/#1}{\underline{#1}}}

\DeclareMathOperator{\hv}{\mathrm{hv}}
\DeclareMathOperator{\lv}{\mathrm{lv}}
\DeclareMathOperator{\V}{\mathrm{v}}
\newcommand{\1}{\mathbf{1}}
\providecommand{\abs}[1]{\left\lvert#1\right\rvert}

\usepackage{enumitem}

\begin{document}


\theoremstyle{plain}
\newtheorem{theorem}{Theorem}
\newtheorem{Corollary}[theorem]{Corollary}
\newtheorem{Lemma}[theorem]{Lemma}
\newtheorem{Proposition}[theorem]{Proposition}

\theoremstyle{definition}
\newtheorem{definition}[theorem]{Definition}
\newtheorem{Example}[theorem]{Example}
\newtheorem{conjecture}[theorem]{Conjecture}

\theoremstyle{remark}
\newtheorem{remark}[theorem]{Remark}

\begin{center}
\vskip 1cm{\LARGE\bf Chains with Small Intervals \\ 
\vskip .13in
in the Lattice of Binary Paths}
\vskip 1cm
I. Tasoulas, K. Manes, A. Sapounakis, and P. Tsikouras\\
Department of Informatics\\
University of Piraeus\\
18534 Piraeus\\
Greece\\
\href{mailto:jtas@unipi.gr}{\tt jtas@unipi.gr} \\
\href{mailto:kmanes@unipi.gr}{\tt kmanes@unipi.gr}\\
\href{mailto:arissap@unipi.gr}{\tt arissap@unipi.gr} \\
\href{mailto:pgtsik@unipi.gr}{\tt pgtsik@unipi.gr} \\
\end{center}

\begin{abstract}
We call an interval $[x,y]$ in a poset {\em small} if $y$ is the join of some elements covering $x$. In this paper, we study the chains of paths from a given arbitrary (binary) path $P$ to the maximum path having only small intervals. More precisely, we obtain and use several formulas for the enumeration of chains having only small intervals and minimal length. For this, we introduce and study the notions of filling and degree of a path, giving in addition some related statistics. 
\end{abstract}

\section{Introduction}\label{section:intro}

Let $\mathcal{P}_n$ be the set of all (binary) paths $P$ of length $|P| = n$, i.e., lattice paths $P = p_1 p_2 \cdots p_n$, starting from the origin of a pair of axes, where each {\em step} $p_i$, $i \in [n]$, is either an {\em upstep} $u = (1,1)$ or a {\em downstep} $d = (1,-1)$, connecting two consecutive points of the path. We denote by $|P|_u$ (resp., $|P|_d$) the number of upsteps (resp., downsteps) of $P$. An {\em ascent} (resp., {\em descent}) of $P$ is a maximal sequence of $u$'s (resp., $d$'s) in $P$. A {\em peak} (resp., {\em valley}) of the path is the last point of an ascent (resp., descent). Clearly, every peak (resp., valley) is either the middle point of an occurrence of $ud$ (resp., $du$), or the endpoint of an occurrence of $u$ (resp., $d$) at the end of the path. The {\em height} of a point of the path $P$ is its $y$-coordinate. We denote by $\lv(P)$ (resp., $\hv(P)$) the height of the lowest (resp., highest) valley of $P$. A {\em low} valley of $P$ is a valley of $P$ with height $\lv(P)$. We set $\mathcal{P} = \bigcup_{n \ge 0} \mathcal{P}_n$, where $\mathcal{P}_0$ consists of only the empty path $\varepsilon$ (the path which has no steps). 

A {\em Dyck path} is a path that starts and ends at the same height and lies weakly above this height. In this paper, we will denote Dyck paths using lower case letters. The set of Dyck paths of length $2n$ is denoted by $\mathcal{D}_n$, and we set $\mathcal{D} = \bigcup_{n \ge 0} \mathcal{D}_n$, where $D_0 = \{\varepsilon\}$. It is well known that $|D_n| = C_n$, where
$C_n = \frac{1}{n+1} \binom{2n}{n}$ is the $n$-th Catalan number, (sequence \seqnum{A000108} in OEIS \cite{OEIS}). Every Dyck path of the form $u^n d^n$, where $n \ge 0$, is called {\em pyramid}. A {\em Dyck prefix} (resp., {\em Dyck suffix}) is a path which is a prefix (resp., suffix) of a Dyck path. Every non-initial point of a Dyck prefix having height zero is called {\em return}. A {\em prime} Dyck path is a Dyck path with only one return point. It is well known that every non-empty Dyck path $a$ is the product of prime Dyck paths, i.e., $a = u a_1 d u a_2 d \cdots u a_k d$, where $a_i \in \mathcal{D}$, $i \in [k]$. Every Dyck prefix (resp., Dyck suffix) $P$ can be uniquely decomposed in the form $P = a_0 u a_1 \cdots u a_k$ (resp., $P = a_0 d a_1 \cdots d a_k$), where $a_i \in \mathcal{D}$, $i \in [0,k]$, $k \ge 0$. 

A natural (partial) ordering on $\mathcal{P}_n$ is defined via the geometric representation of the paths $P, Q \in \mathcal{P}_n$, where $P \le Q$ whenever $P$ lies (weakly) below $Q$. Obviously, $Q$ covers $P$ whenever $Q$ is obtained from $P$ by turning exactly one valley of $P$ into a peak. 
This ordering is better understood by considering the following alternative encoding of binary paths: Every $P \in \mathcal{P}_n$ can be described uniquely by the sequence $(h_i(P))_{i \in [n]}$ of the heights of its points, so that $P \le Q$ iff $h_i(P) \le h_i(Q)$, $i \in [n]$. Then, the join and meet of $P$, $Q$ are given by: 
\[ h_i(P \vee Q) = \max\{ h_i(P), h_i(Q) \} \textrm{ and } h_i(P \wedge Q) = \min\{ h_i(P), h_i(Q) \}. \]
 From these relations, it follows immediately that the poset $(\mathcal{P}_n, \le)$, or simply $\mathcal{P}_n$, is a finite distributive lattice. Clearly, $\mathcal{P}_n$ is self-dual, with minimum and maximum elements the paths $\mathbf{0}_n = d^n = \underbrace{d d \cdots d}_{n \ \rm{times}}$ and $\1_n = u^n = \underbrace{u u \cdots u}_{n \ \rm{times}}$ respectively.

We note that the length of every maximal chain of the interval $[P,Q]$, where $P = p_1 p_2 \cdots p_n$ and $Q = q_1 q_2 \cdots q_n$, is equal to 
\[ l(P,Q) = \frac{1}{2} \sum\limits_{i=1}^n (h_i(Q) - h_i(P)) = \sum\limits_{i=1}^n (n-i+1) \cdot ([q_i = u] - [p_i = u]), \]
where $[S]$ is the Iverson binary notation, i.e., for every proposition $S$, $[S] = 1$ if $S$ is true, and $0$ if $S$ is false. Hence, the lattice $\mathcal{P}_n$ is graded with rank equal to $\binom{n+1}{2}$, and its rank function is
\[ \rho(P) = \frac{1}{2} \left( \sum\limits_{i=1}^n h_i(P) + \binom{n+1}{2} \right) = \sum\limits_{i=1}^n (n-i+1) [p_i = u]. 
\]

This lattice appears in the literature in various equivalent forms (e.g., binary words \cite[p.\ 92]{KnuthV4A}, subsets of $[n]$ \cite{Lindstrom1970}, permutations of $[n]$ \cite[p.\ 402]{Stanley2011}, partitions of $n$ into distinct parts \cite{Stanley1991}, threshold graphs \cite{MerrisRoby2005}). 
The sublattice $\mathcal{D}_n$ of Dyck paths has been studied by several authors (e.g., \cite{GP1996, STT2006}). Manes et al.\ \cite{MSTT2018} have recently presented a bijection between comparable pairs of paths of this lattice and Dyck prefixes of odd length.

Every path $P$ can be decomposed as
\begin{equation}\label{k_i}
P = u^{k_1}d u^{k_2-k_1} d u^{k_3-k_2} d \cdots u^{k_m-k_{m-1}} d u^{k_{m+1}-k_m},
\end{equation}
where $m=|P|_d$ and $(k_i)_{i \in [m+1]}$ is a non-decreasing sequence of integers.
For $i \in [m]$, the term $k_i$ is the number of $u$'s before the $i$-th downstep of $P$ and $k_{m+1} = |P|_u$. In the sequel, we will write $P=(k_i)_{i \in [m+1]}$ to denote the encoding of $P$ by this sequence. This sequence is an extension of the notion of $P$-sequences defined by Pallo and Racca \cite{PR1985} for binary trees, and used by Germain and Pallo \cite{GP1996} in an equivalent form for Dyck paths, in order to prove that $\mathcal{D}_n$ is a distributive graded lattice. 
It is well known (e.g., see \cite[Theorem 10.7.1]{Krattenthaler}) that, using this encoding, the cardinality of every interval $[P,Q]$ in $\mathcal{P}_n$ can be evaluated for every pair of paths $P,Q$ when the two paths end at the same point (i.e., $|P|_d = |Q|_d$). Indeed, we have
\begin{equation}\label{eq:interval1} 
|[P,Q]| = \det\limits_{i,j \in [m]} \left( \binom{\mu_i - k_j + 1}{j -i +1} \right),
\end{equation}
where $P = (k_i)$, $Q = (\mu_i)$, $i \in [m+1]$.

In order to evaluate the cardinality $|[P,Q]|$ for two paths $P, Q \in \mathcal{P}_n$ that do not end at the same point (i.e., $|Q|_d < |P|_d$), we partition the interval $[P,Q]$ into intervals $[P_i, Q_i]$, $i \in [0, |P|_d - |Q|_d]$, of paths ending at the same point, where $P_i$ (resp., $Q_i$) is the path obtained by turning the last $|P|_d - |Q|_d - i$ $d$'s of $P$ into $u$'s (resp., the last $i$ $u$'s of $Q$ into $d$'s). Thus, we obtain
\begin{equation}\label{eq:interval2} |[P,Q]| = \sum\limits_{i=0}^{|P|_d - |Q|_d} |[P_i, Q_i]|.
\end{equation}

In this work, we will mainly deal with the intervals $[a, u^{|a|/2} d^{|a|/2}]$ and $[P, u^{|P|}]$ for $a \in \mathcal{D}$ and $P \in \mathcal{P}$, using the notation $I(a) = |[a, u^{|a|/2} d^{|a|/2}]|$ and $J(P) = |[P, u^{|P|}]|$.

In this paper we study chains from a certain path $P$ to the maximum path such that each member of the chain is the join of some covers of the previous element, i.e., chains with small intervals only. More formally, we say that a chain (or more generally a multichain) $C: P_0 \le P_1 \le \cdots \le P_k$ in $\mathcal{P}$ has (only) small intervals if $P_i$ is obtained by turning some valleys of $P_{i-1}$ into peaks, for every $i \in [k]$. 

In section~\ref{section:filling}, we introduce and study the notions of filling and degree of a path $P \in \mathcal{P}_n$, which will be used for the evaluation of the number $f(P)$ of minimal $P - \1_n$ chains with small intervals. Apart from this, the filling and the degree are of independent interest and they are related to some interesting statistics. Although Sapounakis et al.\ \cite{STT2006} have already defined them for the sublattice $\mathcal{D}_n$ of Dyck paths, these new notions are not simple extensions of the old ones. Furthermore, we give a connection between minimal chains with small intervals and the powers of the M\"obius function.

In section~\ref{section:mapf}, which is the main part of this paper, we evaluate the number $f(P)$ for an arbitrary path $P \in \mathcal{P}_n$. We do this by producing several formulas concerning special classes of paths, the combination of which completes the general case. Finally, we show that for specific classes of paths the map $f$ is related to the zeta function.

\section{Filling and degree of a path}\label{section:filling}

For every path $P \in \mathcal{P}_n \setminus \{\1_n\}$, we call the join of all elements covering $P$ {\em filling} of $P$, and we denote it by $\widetilde{P}$. We also define $\widetilde{\1_n} = \1_n$, for $n \ge 0$. Obviously, the filling of $P$ is obtained by turning every valley of $P$ into a peak. For example, if $P = dduudududdd$, then $\widetilde{P} = duduududddu$.

We note that for every $P, Q \in \mathcal{P}_n$ with $P \le Q < \1_n$ we have that $P < \widetilde{P} \le \widetilde{Q}$. 

It is easy to check that the interval $[P, \widetilde{P}]$ is isomorphic to the Boolean lattice $B_k$, where $k$ is the number of valleys of $P$, for each $P \in \mathcal{P}_n$.

In the following result we characterize the set of fillings of $\mathcal{P}_n$. 
For the proof, we can easily show by induction, using the first valley decomposition, that a path $P$ is a filling iff every valley of $P$ is adjacent to a peak.

\begin{Proposition}
A path $P$ is a filling of some path in $\mathcal{P}$ iff 
$P \ne d$ and it satisfies the following conditions:
\begin{enumerate}[label=\roman{*}),itemsep=0em]
\item $P$ avoids $d^2 u^2$,
\item $P$ does not start with $d u^2$,
\item $P$ does not end with $d^2$.
\end{enumerate} 
\end{Proposition}

In the next result, we use the above characterization in order to enumerate the fillings of $\mathcal{P}_n$.

\begin{Proposition}
The number $a_n$ of fillings of $\mathcal{P}_n$ is the $(n+1)$-th tribonacci number (seq.\ \seqnum{A000213} in OEIS \cite{OEIS}), given by  
\begin{equation}\label{tribeq} a_n = a_{n-1} + a_{n-2} + a_{n-3}, \  n \ge 3,
\end{equation}
and $a_0 = a_1 = 1$, $a_2 = 3$. 
\end{Proposition}

\begin{proof}
Let $\widetilde{\mathcal{P}}_n$ be the set of fillings in $\mathcal{P}_n$. 
For $n = 0,1,2$ we have that $\widetilde{\mathcal{P}}_0 = \{ \varepsilon \}, \ \widetilde{\mathcal{P}}_1 = \{ u \}, \ \widetilde{\mathcal{P}}_2 = \{ uu, ud, du \}$ and hence, $|\widetilde{\mathcal{P}}_0| = |\widetilde{\mathcal{P}}_1| = 1$ and $|\widetilde{\mathcal{P}}_2| = 3$. For $n \ge 3$, $\widetilde{\mathcal{P}}_n$ can be partitioned into the following three sets:
\begin{align*}
\widetilde{\mathcal{P}}_{n,1} & = \{ P \in \widetilde{\mathcal{P}}_n : P \textrm{ starts with $uu$ or $dd$} \}, \\
\widetilde{\mathcal{P}}_{n,2} & = \{ P \in \widetilde{\mathcal{P}}_n : P \textrm{ starts with $ud$} \}, \\
\widetilde{\mathcal{P}}_{n,3} & = \{ P \in \widetilde{\mathcal{P}}_n : P \textrm{ starts with $du$} \}.
\end{align*}
By deleting the first step of each path of $\widetilde{\mathcal{P}}_{n,1}$, we can easily check that $\abs{\widetilde{\mathcal{P}}_{n,1}} = \abs{\widetilde{\mathcal{P}}_{n-1}} = a_{n-1}$. 
Similarly, by deleting the first two (resp., three) steps of every path in $\widetilde{\mathcal{P}}_{n,2}$ (resp., $\widetilde{\mathcal{P}}_{n,3}$), we obtain that $\abs{\widetilde{\mathcal{P}}_{n,2}} = \abs{\widetilde{\mathcal{P}}_{n-2} } = a_{n-2}$ (resp., $\abs{ \widetilde{\mathcal{P}}_{n,3} } = \abs{ \widetilde{\mathcal{P}}_{n-3} } = a_{n-3}$), which gives relation~\eqref{tribeq}. 
\end{proof}

Using the notion of the filling, we restate that a chain $C: P_0 \le P_1 \le \cdots \le P_k$ has small intervals iff $P_i \le \widetilde{P_{i-1}}$, for every $i \in [k]$.

For every path $P \in \mathcal{P}_n$ we define inductively a finite sequence of paths $P^{(i)}$ in $\mathcal{P}_n$, as follows: $P^{(0)} = P$, and $P^{(i)} = \widetilde{P^{(i-1)}}$ whenever $P^{(i-1)} \ne \1_n$. The number $\delta(P)$ for which $P^{(\delta(P))} = \1_n$ is called the {\em degree} of $P$. Clearly, the chain $C: P_0 = P^{(0)} \le P^{(1)} \le \cdots \le P^{(\delta(P))} = \1_n$ is a $P - \1_n$ chain with small intervals and length $\delta(P)$. In the following result, we establish the minimality of $\delta(P)$ with respect to this property.

\begin{Proposition}\label{prop:minlength}
The length of every chain from a path $P \in \mathcal{P}_n$ to $\1_n$, with small intervals is greater than or equal to $\delta(P)$.
\end{Proposition}

\begin{proof}
Let $C: P = P_0 \le P_1 \le \cdots \le P_k = \1_n$ be a chain with $P_i \le \widetilde{P_{i-1}}$, for every $i \in [k]$. For every $j \in [\delta(P)]$, we denote by $i_j$ the greatest element of $[k]$ such that $P_{i_j-1} \le P^{(j-1)}$. It follows that $P^{(j-1)} < P_{i_j} \le \widetilde{P_{i_j-1}} \le P^{(j)}$, so that every interval $[P^{(j-1)}, P^{(j)}]$ contains an element of the chain $C$, giving automatically that $\delta(P) \le k$. 
\end{proof}

We now come to the evaluation of $\delta(P)$ for every path $P \in \mathcal{P}_n$. Clearly, we have $\delta(\1_n) = 0$, $\delta(\mathbf{0}_n) = n-1$ and $\delta(\widetilde{P}) = \delta(P) - 1$ for $P \ne \1_n$. In the general case, we will see that $\delta(P)$ is closely related to $\lv(P)$ (the height of the lowest valley of $P$). Indeed, since $\widetilde{P}$ is obtained by turning all valleys of $P$ into peaks, it follows that the heights of their low valley points will differ by exactly one, i.e., 
$\lv(\widetilde{P}) = \lv(P) + 1$, for $P \ne u^{|P|}, u^{|P|-1} d$. It follows that
\[ \lv(P^{(i)}) = \lv(P^{(i-1)}) + 1, \  1 \le i \le \delta(P) - 1. \]
Summing for all $i$, we obtain that
\[ \lv(P^{(\delta(P) -1)}) = \lv(P^{(0)}) + \delta(P) - 1, \]
so that
\[ |P| - 2 = \lv(P) + \delta(P) - 1, \]
giving the following result.

\begin{Lemma}\label{lemma:degreeformula} For every path $P \in \mathcal{P}$ with $P \ne u^{|P|}$ the degree $\delta(P)$ is given by the formula
\[ \delta(P) = |P| - 1 - \lv(P). \]
\end{Lemma}

\begin{Proposition}
The number of paths of $\mathcal{P}_n$ having degree $k$ equals 
\[ \binom{\min\{n,k\}}{\lfloor\frac{k+2}{2}\rfloor}\] for $1 \le k \le 2n-1$, $n \in \mathbb{N}^*$.
\end{Proposition}

\begin{proof}
Let $\Delta_{n,k}$ be the set of all $P \in \mathcal{P}_n$ with $\delta(P) = k$, and let $M_{n,k}$, $N_{n,k}$ be its subsets of paths which start with $u$ and $d$ respectively. 

We first prove that for $n \ge 2$ we have that
\begin{equation}\label{Mkn_formula} 
\abs{M_{n,k}} = \abs{\Delta_{n-1,k}},
\end{equation} 
and
\begin{equation}\label{Nkn_formula}
\abs{N_{n,k}} = \begin{cases} 
0, & \textrm{ if $k < n$}; \\ 
\abs{\mathcal{DP}_{n-1}}, & \textrm{ if $k = n$}; \\
\abs{\Delta_{n-1,k-2}}, & \textrm{ if $k > n$},
\end{cases}
\end{equation}
where $\mathcal{DP}_{n}$ is the set of Dyck prefixes of length $n$, which is well known that it is enumerated by the binomial $\binom{n}{\lfloor n/2 \rfloor}$, (seq.\ \seqnum{A001405} in OEIS \cite{OEIS}). Indeed, by deleting the first $u$ from each path of $M_{n,k}$ we obtain a path in $\Delta_{n-1,k}$, giving a bijection between the sets $M_{n,k}$ and $\Delta_{n-1,k}$ which justifies relation \eqref{Mkn_formula}.

On the other hand, for the proof of relation \eqref{Nkn_formula}, we consider the following cases:

\begin{enumerate}[label=(\roman{*})]

\item if $k < n$, then $\lv(P) \ge 0$, so that every $P \in \Delta_{n,k}$ starts with $u$, giving $N_{n,k} = \emptyset$.

\item $k \ge n$; by deleting the first $d$ from each path $P \in N_{n,k}$ we obtain a path $Q$ such that if $k = n$, then $\lv(P) = -1$, or equivalently $Q \in \mathcal{DP}_{n-1}$, whereas if $k > n$, then $\delta(Q) = k-2$, i.e., $Q \in \Delta_{n-1,k-2}$; this gives a bijection between the sets $N_{n,k}$ and $\mathcal{DP}_{n-1}$ if $k = n$, and between the sets $N_{n,k}$ and $\Delta_{n-1,k-2}$ if $k > n$, justifying relation \eqref{Nkn_formula}.

\end{enumerate}

Clearly, the result holds for $n = 1$, whereas for $n \ge 2$, using  \eqref{Mkn_formula} and \eqref{Nkn_formula}, we proceed by induction on $n$.

If $k < n$, then $\abs{\Delta_{n,k}} = \abs{\Delta_{n-1,k}} = \binom{k}{\lfloor \frac{k+2}{2} \rfloor}$. 

If $k = n$, then 
$\abs{\Delta_{n,n}} 
 = \abs{\Delta_{n-1,n}} + \abs{\mathcal{DP}_{n-1}} 
 = \binom{\min\{n-1,n\}}{\lfloor \frac{n+2}{2} \rfloor} + \binom{n-1}{\lfloor \frac{n-1}{2} \rfloor} 
 = \binom{n-1}{\lfloor \frac{n+2}{2} \rfloor} + \binom{n-1}{\lfloor \frac{n-1}{2} \rfloor}  = \binom{n}{\lfloor \frac{n+2}{2} \rfloor}$.

If $k > n$, then
$\abs{\Delta_{n,k}}
 = \abs{\Delta_{n-1,k}} + \abs{\Delta_{n-1,k-2}} = \binom{n-1}{\lfloor \frac{k+2}{2} \rfloor} + \binom{n-1}{\lfloor \frac{k}{2} \rfloor} = \binom{n}{\lfloor \frac{k+2}{2} \rfloor}$. \qedhere
\end{proof}

The minimal $P - \1_n$ chains with small intervals which, according to Proposition \ref{prop:minlength}, must have length equal to $\delta(P)$, are closely related to the powers of the M\"obius function $\mu$ of $\mathcal{P}$.
Indeed, if $f$ is the map on $\mathcal{P}$ defined by
\[ f(P) = \# \ P - \1_n \textrm{ chains of length $\delta(P)$ with small intervals}, \]
then we have the following result.

\begin{Proposition}
For every path $P \in \mathcal{P}$ we have that
\[ \mu^k(P, \1_{|P|}) = 
\begin{cases}
0, & \textrm{ if $k < \delta(P)$}; \\
(-1)^{l(P, \1_{|P|})} f(P), & \textrm{ if $k = \delta(P)$.}
\end{cases} \]
\end{Proposition}

\begin{proof}
Since $\mathcal{P}$ is a distributive locally finite lattice, it follows (e.g., see \cite[Proposition 3.7, p.\ 90]{BP1995}) that the M\"obius function of $\mathcal{P}$ has the following formula
\[ \mu(P,Q) = \begin{cases}
(-1)^{l(P,Q)}, & \textrm{ if $P \le Q \le \widetilde{P}$}; \\
0, & \textrm{ otherwise. }
\end{cases} \]
Furthermore, for $k \ge 1$, we can easily check that
\[ \mu^k(P,Q) = \sum\limits \prod\limits_{i=1}^k \mu(P_{i-1}, P_i), \]
where the sum is taken over all multichains $C: P = P_0 \le P_1 \le \cdots \le P_k = Q$ of length $k$ with small intervals, so that
\[ \mu^k(P,Q) = (-1)^{l(P,Q)} \cdot (\textrm{\# $P-Q$ multichains of length $k$ with small intervals}). \]
Then, for $Q = \1_{|P|}$, using Proposition \ref{prop:minlength}, the required formula follows automatically.
\end{proof}

\section{Counting minimal chains with small intervals}\label{section:mapf}

In this section, we evaluate the number $f(P)$ of minimal $P-\1_n$ chains of length $\delta(P)$ with small intervals, for every path $P \in \mathcal{P}$. We will use the notation $P^\prime$ (resp., $P^*$) for the path obtained by turning every low valley of $P$ (resp., every valley of $P$, except the last one if $P$ ends with $d$) into a peak.

Note that $f(\1_n) = 1$, for every $n \ge 0$. Clearly, we have $f(u P) = f(P)$ for every $P \in \mathcal{P}$, so that it is enough to evaluate $f(P)$ when $P$ is a Dyck prefix. In the following result, we give a recursive formula for the map $f$. 

\begin{Proposition}\label{prop:basicrec}
For every path $P \in \mathcal{P}_n \setminus \{\1_n\}$ we have that
\[ f(P) = \sum\limits_{Q \in [P^\prime, \widetilde{P}]} f(Q). \]
\end{Proposition}

\begin{proof}
Let $C : P = P_0 \le P_1 \le \cdots \le P_k = \1_n$, where $k = \delta(P)$, be a $k$-chain from $P$ to $\1_n$ with small intervals. Then, since $P_{i-1} \le P_i \le \widetilde{P_{i-1}}$ for every $i \in [k]$, we can easily prove by induction that $P_i^{(k-i)} = \1_n$, i.e., $\delta(P_i) = k-i$, for every $i \in [k]$. In particular, $\delta(P_1) = k-1$, which by Lemma~\ref{lemma:degreeformula} gives that $\lv(P_1) = \lv(P) +1$. This shows that $P_1 \in [P^\prime, \widetilde{P}]$. Moreover, if we delete $P$ from $C$ we obtain a $(k-1)$-chain from $P_1$ to $\1_n$ with small intervals.

On the other hand, given $Q \in [P^\prime, \widetilde{P}]$, by adding $P$ in the beginning of every $\delta(Q)$-chain from $Q$ to $\1_n$ with small intervals, we obtain a $\delta(P)$-chain from $P$ to $\1_n$ with small intervals.

Thus, the result follows automatically by decomposing the $\delta(P)$-chains from $P$ to $\1_n$ with small intervals according to their second member.
\end{proof}

\begin{Corollary}
For every $P \in \mathcal{P}$ we have that
\[ f(P) = 1 \textrm{ iff all valleys of $P$ have the same height.} \] 
\end{Corollary}

Every path $P \in \mathcal{P}$ can be decomposed (not necessarily uniquely) 
as a product of a Dyck suffix $P_1$ followed by a Dyck prefix $P_2$. In the following result we give a recursive formula of $f$ that utilizes this decomposition. 

\begin{Proposition}\label{prop:suffixprefix}
If $P = P_1 P_2$, where $P_1$ is a Dyck suffix and $P_2$ is a Dyck prefix, then $f(P) = f(P_1) f(d P_2)$.
\end{Proposition}

\begin{proof} Without loss of generality we may assume that $P$ starts with $d$, $P_1 \ne d$ and $P_2 \ne \varepsilon$.

We use induction with respect to the length and to the (dual) partial order, i.e., assuming that the result holds for every path $Q$ (that starts with $d$) for which $|Q| < |P|$, or $|Q| = |P|$ and $Q > P$, we will prove that the result holds for $P$.

We decompose $P_1$, $P_2$ as follows:
\[ P_1 = R_1 a d, \qquad P_2 = u b R_2, \]
where $R_1$ (resp., $R_2$) is a Dyck suffix (resp., Dyck prefix), $a$ (resp., $b$) may be either empty, or a path that starts with $d$ (resp., ends with $u$), ends at height $0$ and it is bounded by the line $y = -1$ (see Figure \ref{fig:SP}). 

\begin{figure}[ht]
\begin{center}
\psset{unit=1em}
\begin{pspicture}(-2,0)(35,4)
\rput[r](0,1){$P=$}
\psellipticarc[](2,1)(9,2.15){0}{120} \psline[linestyle=dashed](1,3)(4,1)(11,1) \rput[b](6,1.5){$R_1$}
\psellipticarc[](13,1)(2,1.5){0}{180}  \psline[linewidth=2pt](11,1)(12,0) 
\psline[linestyle=dashed](12,0)(14,0)(15,1) \psline[linewidth=2pt](15,1)(16,0)(17,1)  \rput[b](13,1){$a$}
\psellipticarc[](19,1)(2,1.5){0}{180} \psline[linestyle=dashed](17,1)(18,0)(20,0) 
\psline[linewidth=2pt](20,0)(21,1) \rput[b](19,1){$b$}
\psellipticarc[](30,1)(9,2.15){60}{180} \psline[linestyle=dashed](21,1)(28,1)(31,3) \rput[b](26,1.5){$R_2$}
\psline[linestyle=dotted](1,0)(31,0)
\rput[l](32,0){$y=-1$}
\end{pspicture}
\end{center}
\caption{The decomposition of $P=P_1P_2$}
\label{fig:SP}
\end{figure}
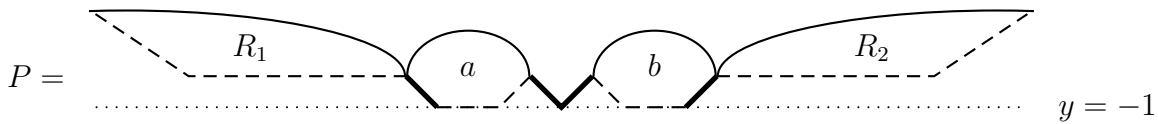

Then, we have that
\[ P_1^\prime = R_1 a^\prime u, \qquad  (d P_2)^\prime = u d b^\prime R_2, \qquad P^\prime = R_1 a^\prime u d b^\prime R_2, \]
\[ \widetilde{P_1} = R_1^* a^* u, \qquad \widetilde{d P_2} = u d \ \widetilde{b} \ \widetilde{R_2}, \qquad \widetilde{P} = R_1^* a^* u d \ \widetilde{b} \  \widetilde{R_2}. \] 
It follows that every $Q \in [P^\prime, \widetilde{P}]$ can be uniquely decomposed as
\[ Q = Q_1 u d Q_2, \]
where $Q_1 \in [R_1 a^\prime, R_1^* a^*]$, $Q_2 \in [b^\prime R_2, \widetilde{b} \widetilde{R_2}]$.
Clearly, since $Q_1$ is a Dyck suffix and $Q_2$ is a Dyck prefix with $Q > P$ and $|Q_1|, |Q_2| \le |P| - 2$, we have that 
\begin{align*} f(Q) 
& = f(Q_1 u d Q_2) = f(Q_1 ud) f(d Q_2) = f(Q_1) f(dud) f(u d Q_2) \\
& = f(Q_1) f(du) f(ud Q_2) = f(Q_1 u) f(u d Q_2), \end{align*}
and therefore
\begin{align*}
f(P) 
& = \sum\limits_{Q \in [P^\prime, \widetilde{P}]} f(Q) = \sum\limits_{\stackrel{Q_1 \in [R_1 a^\prime, R_1^* a^*]}{Q_2 \in [b^\prime R_2, \widetilde{b} \widetilde{R_2}]}} f(Q_1 u) f(u d Q_2) \\
& = \sum\limits_{Q_1 u \in [R_1 a^\prime u, R_1^* a^* u]} f(Q_1 u) \sum\limits_{u d Q_2 \in [u d b^\prime R_2, ud \widetilde{b} \widetilde{R_2}]}  f(u d Q_2) \\
& = \sum\limits_{Q_1 u \in [P_1^\prime, \widetilde{P_1}]} f(Q_1 u) \sum\limits_{u d Q_2 \in [(d P_2)^\prime, \widetilde{d P_2}]} f(u d Q_2)\\
& = f(P_1) f(d P_2),
\end{align*}
completing the proof.
\end{proof}

\begin{Corollary}
If $P$ is a Dyck prefix with at least one return point, then $f(dP) = f(P)$.
\end{Corollary}

\begin{proof} We use induction with respect to the length of the path.  Clearly, since $P$ has at least one return point, it can be written as $P = u a d R$, where $a \in \mathcal{D}$ and $R$ is a Dyck prefix. Then, using Propositions \ref{prop:basicrec}, \ref{prop:suffixprefix} and the induction hypothesis, we have that
\begin{align*}
f(d u a d)
& = \sum\limits_{Q \in [u d a u, u d a^* u]} f(Q)  = \sum\limits_{b \in [a, a^*]} f( u d b u)  = \sum\limits_{b \in [a, a^*]} f( u d b) f(d u) \\
& = \sum\limits_{b \in [a, a^*]} f( u b) f( d u) = \sum\limits_{b \in [a, a^*]} f( u b u) = \sum\limits_{Q \in [u a u, u a^* u]} f(Q) = f(u a d).
\end{align*} 
It follows that
$f(d P) = f(d u a d) f(d R) = f(u a d) f(d R) = f(P)$. \end{proof}

 From the two previous results it follows that the map $f$ on the set of Dyck paths is multiplicative. Moreover, since every Dyck prefix can be uniquely decomposed in the form $a Q$, where $a = \varepsilon$ or $a = u a_1 d \cdots u a_k d $, $a_i \in \mathcal{D}$, $i \in [k]$ and $Q = \varepsilon$ or $Q = u P$ for some Dyck prefix $P$, for the evaluation of $f$ it is enough to restrict ourselves to the two cases $f(u a d)$ and $f(d u P)$, where $a \in \mathcal{D}$ and $P$ is a Dyck prefix.  For this, we introduce a new kind of multichains for Dyck paths, based on the heights of the valleys of the paths. 

We say that a multichain of Dyck paths $C : \sigma_0 \le \sigma_1 \le \cdots \le \sigma_h$, where $h = \hv(\sigma_0)$ (the height of the highest valley of $\sigma_0$) is {\em of type $\V$} iff for every $j \in [h]$ the paths $\sigma_j$, $\sigma_{j-1}$ have the same valleys at every height at most $h - j$. 

\begin{Example}\label{example:vchain} The multichain $C : \sigma_0 \le \sigma_1 \le \sigma_2 \le \sigma_3$ with
\begin{align*}
\sigma_0 & = u^2 d u^2 d u^2 d^2 u^2 d^3 u d^2 u^3 d u^2 d u d^2 u d u d^3, \\
\sigma_1 & = u^2 d u^2 d u^2 d^2 u^2 d^3 u d^2 u^3 d u^3 d^3 u d u d^3, \\
\sigma_2 & = u^2 d u^4 d^2 u d u d^3 u d^2 u^5 d u^2 d^4 u d^3, \\
\sigma_3 & = u^4 d u^2 d^2 u d u d u d^4 u^6 d u d^4 u d^3
\end{align*}
is of type $\V$, whereas the multichain $\sigma_0 \le \sigma_1 \le \sigma_2 \le s$ with $s = u^5 d u^3 d^3 u d^3 u d u^4 d u d u d^3 u d^4$ 
is not of type $\V$ (see Figure~\ref{fig:vexample}). 
\end{Example}

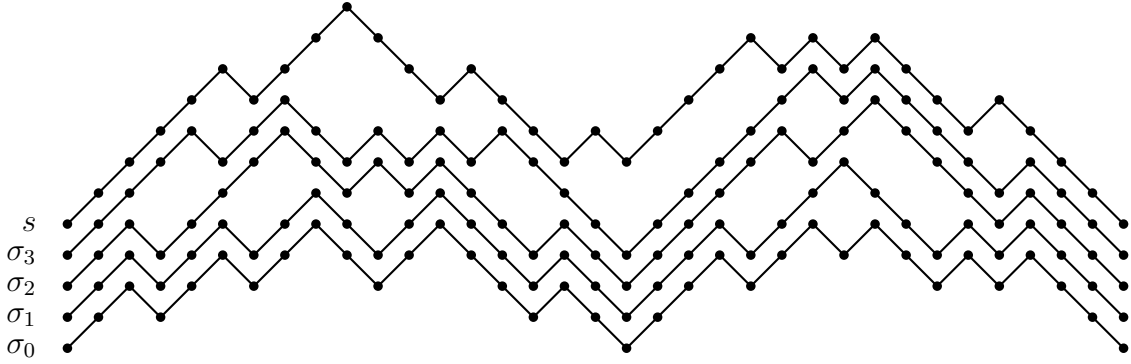
\begin{figure}[ht]
\begin{center}
\psset{unit=1em}
\begin{pspicture}(-1,0)(34,11)

\psline[showpoints=true]
(0,0)(1,1)(2,2)(3,1)(4,2)(5,3)(6,2)(7,3)(8,4)(9,3)(10,2)(11,3)(12,4)(13,3)(14,2)(15,1)(16,2)(17,1)(18,0)(19,1)(20,2)
(21,3)(22,2)(23,3)(24,4)(25,3)(26,4)(27,3)(28,2)(29,3)(30,2)(31,3)(32,2)(33,1)(34,0) 
\rput[r](-1,0){$\sigma_0$}

\psline[showpoints=true]
(0,1)(1,2)(2,3)(3,2)(4,3)(5,4)(6,3)(7,4)(8,5)(9,4)(10,3)(11,4)(12,5)(13,4)(14,3)(15,2)(16,3)(17,2)(18,1)(19,2)(20,3)
(21,4)(22,3)(23,4)(24,5)(25,6)(26,5)(27,4)(28,3)(29,4)(30,3)(31,4)(32,3)(33,2)(34,1)   
\rput[r](-1,1){$\sigma_1$}

\psline[showpoints=true](0,2)(1,3)(2,4)(3,3)(4,4)(5,5)(6,6)(7,7)(8,6)(9,5)(10,6)(11,5)(12,6)(13,5)(14,4)(15,3)(16,4)(17,3)(18,2)
(19,3)(20,4)(21,5)(22,6)(23,7)(24,6)(25,7)(26,8)(27,7)(28,6)(29,5)(30,4)(31,5)(32,4)(33,3)(34,2)   
\rput[r](-1,2){$\sigma_2$}

\psline[showpoints=true](0,3)(1,4)(2,5)(3,6)(4,7)(5,6)(6,7)(7,8)(8,7)(9,6)(10,7)(11,6)(12,7)(13,6)(14,7)
(15,6)(16,5)(17,4)(18,3)(19,4)(20,5)(21,6)(22,7)(23,8)(24,9)(25,8)(26,9)(27,8)(28,7)(29,6)(30,5)(31,6)(32,5)(33,4)(34,3)  
\rput[r](-1,3){$\sigma_3$}

\psline[showpoints=true](0,4)(1,5)(2,6)(3,7)(4,8)(5,9)(6,8)(7,9)(8,10)(9,11)
(10,10)(11,9)(12,8)(13,9)(14,8)(15,7)(16,6)(17,7)(18,6)
(19,7)(20,8)(21,9)(22,10)(23,9)(24,10)(25,9)(26,10)(27,9)(28,8)(29,7)(30,8)(31,7)(32,6)(33,5)(34,4)  
\rput[r](-1,4){$s$}
\end{pspicture}
\end{center}
\caption{The multichain of example \ref{example:vchain}}\label{fig:vexample}
\end{figure}

For $a, s \in \mathcal{D}$ with $a \le s$ we denote by $\V(a,s)$ the number of all $a-s$ multichains of type $\V$. Clearly, we have $\V(a,s) \ne 0$ iff $a, s$ have exactly the same low valleys. Furthermore, we can easily check that
\[ 
\V( u a_1 d \cdots u a_k d, u s_1 d \cdots u s_k d) = \prod\limits_{i=1}^k \V(u a_i d, u s_i d) 
\]
when $a_i \le s_i$ for every $i \in [k]$, and that
\begin{equation}\label{eq:vprime} 
\V(uad, usd) = \sum\limits_{t \le s} \V(a, t). \end{equation}

In the following Proposition, we give an alternative formula of $\V$ on the pairs of prime Dyck paths, which will be used in the sequel.

\begin{Proposition}\label{prop:vbijection}
For every $a, s \in \mathcal{D}$ with $a \le s$ we have that
\[ \V(uad, usd) = \sum\limits_{w \in [a, a^*]} \V(w,s). \]
\end{Proposition}

\begin{proof}
In view of formula \eqref{eq:vprime}, it is enough to construct a bijection between the set of $a - t$ multichains of type $\V$ for all $t \le s$, and the set of $w - s$ multichains of type $\V$ for all $w \in [a, a^*]$.

A key property of Dyck paths that we will use throughout the proof is the following: If two Dyck paths $t, r$ of length $2n$ have the same valleys for every height at most $j$, then they coincide up to height $j+1$, and if $t \le r$ then every valley of $r$ at height $j+1$ is necessarily also a valley of $t$.
Note also that given a multichain $\sigma_0 \le \sigma_1 \le \cdots \le \sigma_h$ of type $\V$ and if $0 \le j \le i \le h$, then $\sigma_j, \sigma_i$ have the same valleys at every height at most $h-i$.

In order to exhibit the required bijection, for a multichain $a = \sigma_0 \le \sigma_1 \le \cdots \le \sigma_h = t$ of type $\V$ with $t \le s$, we set $\sigma_{h+1} = s$ and we first construct a sequence of Dyck paths $\tau_i$, $i \in [0,h+1]$, where $\tau_i$ is the Dyck path obtained by turning into peaks all valleys of $\sigma_i$ at height $j$ that are not also valleys of $\sigma_{h+1-j}$, for every $j \le h-i$. Clearly, we have $\tau_0 \in [a, a^*]$ and $\tau_{h+1} = s$.

Furthermore, $\tau_{i-1} \le \tau_i$ for every $i \in [h+1]$. Indeed, all valleys of $\sigma_{i-1}$, $\sigma_i$ at height $j \le h-i$ which turn into peaks for the construction of $\tau_{i-1}$, $\tau_i$ are the same, whereas the valleys of $\sigma_{i-1}$ at height $h-i+1$ that turn into peaks are not valleys of $\sigma_{h+1-(h-i+1)} = \sigma_i$, so that $\sigma_i$ (let alone $\tau_i$) passes at least two units above these valleys and hence, $\tau_{i-1}$ is weakly below $\tau_i$. 

Set $w = \tau_0$; hence, $w \in [a,a^*]$. Moreover, $\tau_{h+1-k} = w$, where $k = \hv(w)$. Indeed, clearly $\sigma_0, \sigma_{h+1-k}$ have the same valleys at every height at most $k-1$. It follows that the valleys of $\tau_0, \tau_{h+1-k}$ that are either at height at most $k-1$, or at height $k$ and have been created by the above construction are the same. Furthermore, if there exists a valley of $\tau_0$ at height $k$ that has not been created by the construction, then this valley is also a valley of both $\sigma_0$, $\sigma_{h+1-k}$ so that it is also a valley of $\tau_{h+1-k}$. Thus, since the height of the highest valley of $\tau_0$ is $k$, the paths $\tau_0$, $\tau_{h+1-k}$ have the same valleys, which gives $\tau_0 = \tau_{h+1-k}$.

We define $s_i = \tau_{i + (h+1-k)}$, $i \in [0,k]$. It is easy to check that $w = s_0 \le s_1 \le \cdots \le s_k = s$ is a multichain of type $\V$. 

For the converse we note that the peaks of $\tau_i$ at height $j+2$ generated according to the above construction from $\sigma_i$ for $j \le h-i$, are exactly these peaks of $\tau_i$ that are peaks of $\tau_0$ and not of $a$.

Now, for a multichain $w = s_0 \le s_1 \le \cdots \le s_k = s$ with $\hv(w) = k$, we define a multichain $w = \tau_0 \le \tau_1 \le \cdots \le \tau_{h+1} = s$ with $\tau_i = \begin{cases} s_0, & i \in [0,h+1-k]; \\ s_{i - (h+1-k)}, & i \in [h+2-k,h+1]. \end{cases}$

Finally, we define $\sigma_i^\prime$, $i \in [0,h]$ to be the Dyck path obtained by turning all peaks of $\tau_i$ at height $j+2$ that are also peaks of $w$ but not peaks of $a$ into valleys, for every $j \le h-i$.

Clearly, we have $\sigma_0^\prime = a$ and $\sigma_h^\prime \le s$.

Furthermore, we can analogously prove that $a = \sigma_0^\prime \le \sigma_1^\prime \le \cdots \le \sigma_h^\prime \le s$ is a multichain of type $\V$ and $\sigma_i = \sigma_i^\prime$ for every $i \in [0,h]$. 
\end{proof}

In order to illustrate the bijection in the proof of Proposition \ref{prop:vbijection} we give the following example.

\begin{Example}\label{example:vbijection}
Let $a = u^2 d u^2 d u^2 d^2 u^2 d^3 u d^2 u^3 d u^2 d u d^2 u d u d^3$,
$t = u^4 d u^2 d^2 u d u d u d^4 u^6 d u d^4 u d^3$ and $s = u^5 d u^3 d^3 u d^3 u d u^4 d u d u d^3 u d^4$. For the multichain $a = \sigma_0 \le \sigma_1 \le \sigma_2 \le \sigma_3 = t$ of Example \ref{example:vchain} (see Figure \ref{fig:vexample}), using the construction in the proof of Proposition~\ref{prop:vbijection}, we have that
\[ w = \tau_0 = \tau_1 = u^3 d u^2 d u d u d u d^2 u d^2 u d u^3 d u^2 d^2 u d^2 u d^3, \]
\[ \tau_2 = u^3 d u^3 d^2 u d u d^2 u d^2 u d u^4 d u^2 d^4 u d^3, \]
\[ \tau_3 = u^4 d u^2 d^2 u d u d u d^3 u d u^5 d u d^4 u d^3 \]
and
\[ \tau_4 = s. \]
Then, the multichain $w = s_0 \le s_1 \le s_2 \le s_3 = s$, where $s_i = \tau_{i+1}$, $i \in [3]$, is the corresponding $w-s$ multichain of type $\V$  (see Figure \ref{fig:vbijectionexample}). \end{Example}

\begin{figure}[ht]
\begin{center}
\psset{unit=1em}
\begin{pspicture}(-4,0)(34,10)

\psline[showpoints=true](0,0)(1,1)(2,2)(3,3)(4,2)(5,3)(6,4)(7,3)(8,4)(9,3)(10,4)(11,3)(12,4)(13,3)(14,2)(15,3)(16,2)(17,1)(18,2)(19,1)
(20,2)(21,3)(22,4)(23,3)(24,4)(25,5)(26,4)(27,3)(28,4)(29,3)(30,2)(31,3)(32,2)(33,1)(34,0) 
\rput[r](-1,0){$\tau_0=\tau_1$}

\psline[showpoints=true](0,1)(1,2)(2,3)(3,4)(4,3)(5,4)(6,5)(7,6)(8,5)(9,4)(10,5)(11,4)(12,5)(13,4)(14,3)(15,4)(16,3)(17,2)(18,3)(19,2)
(20,3)(21,4)(22,5)(23,6)(24,5)(25,6)(26,7)(27,6)(28,5)(29,4)(30,3)(31,4)(32,3)(33,2)(34,1) 
\rput[r](-1,1){$\tau_2$}

\psline[showpoints=true](0,2)(1,3)(2,4)(3,5)(4,6)(5,5)(6,6)(7,7)(8,6)(9,5)(10,6)(11,5)(12,6)(13,5)(14,6)(15,5)(16,4)(17,3)(18,4)(19,3)
(20,4)(21,5)(22,6)(23,7)(24,8)(25,7)(26,8)(27,7)(28,6)(29,5)(30,4)(31,5)(32,4)(33,3)(34,2) 
\rput[r](-1,2){$\tau_3$}

\psline[showpoints=true](0,3)(1,4)(2,5)(3,6)(4,7)(5,8)(6,7)(7,8)(8,9)(9,10)(10,9)(11,8)(12,7)(13,8)
(14,7)(15,6)(16,5)(17,6)(18,5)(19,6)(20,7)(21,8)(22,9)(23,8)(24,9)(25,8)(26,9)(27,8)(28,7)(29,6)(30,7)(31,6)(32,5)(33,4)(34,3) 
\rput[r](-1,3){$s=\tau_4$}
\end{pspicture}
\end{center}
\caption{The multichain produced in Example \ref{example:vbijection} from the multichain of type $\V$ of Example \ref{example:vchain}}
\label{fig:vbijectionexample}
\end{figure}
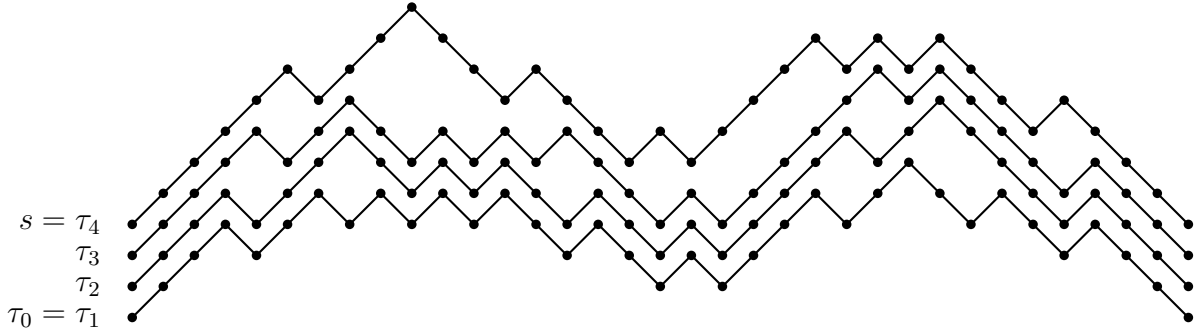

In the following result, we give a formula for the evaluation of $f$ on prime Dyck paths. In the proof, we use the following obvious consequence of Proposition \ref{prop:basicrec}:
\begin{equation}\label{eq:basicrecprimedyck}
f(u a d) = \sum\limits_{w \in [a, a^*]} f(w), \textrm{ for every Dyck path $a$.}
\end{equation}

\begin{Proposition}\label{prop:fprimedyck}
For every Dyck path $a \in \mathcal{D}$ we have that 
\begin{equation}\label{eq:fprimedyck} 
f(uad) = \sum\limits_{s \ge a} \V(a,s) I(s).  
\end{equation} 
\end{Proposition}

\begin{proof}
We prove the required formula by induction on the length of $a$. 

Clearly, it holds for $a = \varepsilon$. Now, let $a = u a_1 d u a_2 d \cdots u a_k d$, where $a_i \in \mathcal{D}$, for $i \in [k]$, $k \in \mathbb{N}^*$. Each $w \in [a, a^*]$ has a common low valley with $a$. Assume that the $r$-th low valley of $a$, $r \in [k]$, is the leftmost such valley. Then, $w = u w_1 u d w_2 \cdots u d w_r d \beta$, where $w_i \in [a_i, a_i^*]$, for $i \in [r]$, $\beta \in [b_r, b_r^*]$ and $b_r = u a_{r+1} d u a_{r+2} d \cdots u a_k d$, $r \in [k]$; (note that $b_k = \varepsilon$). Then, using the induction hypothesis, equation \eqref{eq:basicrecprimedyck} and Proposition \ref{prop:vbijection}, we have that
\begin{align}
f(uad) & = \sum\limits_{w \in [a, a^*]} f(w) = \sum\limits_{r=1}^k \sum\limits_{\stackrel{w_i \in [a_i, a_i^*]}{i \in [r]}} f(u w_1 u d w_2 \cdots u d w_r d) \sum\limits_{\beta \in [b_r, b_r^*]} f(\beta) \nonumber \\
& = \sum\limits_{r=1}^k \sum\limits_{\stackrel{w_i \in [a_i, a_i^*]}{i \in [r]}} \sum\limits_{s \ge w_1 u d w_2 \cdots u d w_r} \V(w_1 ud w_2 \cdots ud w_r, s) I(s) f(u b_r d) \nonumber \\
& = \sum\limits_{r=1}^k \sum\limits_{\stackrel{w_i \in [a_i, a_i^*]}{i \in [r]}} \sum\limits_{\stackrel{s_i \ge w_i}{i \in [r]}} \left( \prod\limits_{i=1}^r \V(w_i, s_i) \right) I(s_1 u d s_2 \cdots u d s_r) f(u b_r d) \nonumber
\end{align}

\begin{align}\label{eq:fuad1}
& = \sum\limits_{r=1}^k \sum\limits_{\stackrel{s_i \ge a_i}{i \in [r]}} \left( \prod\limits_{i=1}^r \sum\limits_{w_i \in [a_i, a_i^*]} \V(w_i, s_i) \right) I(s_1 ud s_2 \cdots ud s_r) f(u b_r d) \nonumber \\
& = \sum\limits_{r=1}^k \sum\limits_{\stackrel{s_i \ge a_i}{i \in [r]}} \left( \prod_{i=1}^r \V( u a_i d, u s_i d) \right) I(s_1 u d s_2 \cdots u d s_r) f(u b_r d).
\end{align}

On the other hand, given a path $s \in \mathcal{D}$ with $s \ge a$, having the same low valleys with $a$, we set $s = u s_1 d u s_2 d \cdots u s_k d$, where $s_i \in \mathcal{D}$ and $s_i \ge a_i$ for every $i \in [k]$. Clearly, as before, we can decompose every path $c \in [s, u^{|s|/2} d^{|s|/2}]$ with respect to the leftmost common low valley with $a$, i.e., $c = u \phi d \chi$ where $\phi, \chi \in \mathcal{D}$, $\phi \ge s_1 u d s_2 \cdots u d s_r$ and $\chi \ge u s_{r+1} d u s_{r+2} d \cdots u s_k d$, for some $r \in [k]$.

It follows that
\begin{align}\label{eq:fuad2}
 \sum\limits_{s \ge a} \V(a,s) I(s) 
& = \sum\limits_{r=1}^k \sum\limits_{\stackrel{s_i \ge a_i}{i \in [k]}} \left( \prod\limits_{i=1}^k \V(u a_i d, u s_i d) \right) I(s_1 u d s_2 \cdots u d s_r) I(u s_{r+1} d u s_{r+2} d \cdots u s_k d) \nonumber \\
& = \sum\limits_{r=1}^k \sum\limits_{\stackrel{s_i \ge a_i}{i \in [r]}} \left( \prod_{i=1}^r \V( u a_i d, u s_i d) \right) I(s_1 u d s_2 \cdots u d s_r) \sum\limits_{t \ge b_r} \V(b_r, t) I(t) \nonumber \\
& = \sum\limits_{r=1}^k \sum\limits_{\stackrel{s_i \ge a_i}{i \in [r]}} \left( \prod_{i=1}^r \V( u a_i d, u s_i d) \right) I(s_1 u d s_2 \cdots u d s_r) f(u b_r d).
\end{align}

The required formula follows from \eqref{eq:fuad1} and \eqref{eq:fuad2}.
\end{proof}

\begin{Example} For $a = u^2 d^2 u^3 d u d u d^3$, from formula \eqref{eq:fprimedyck} we have that $f(u^3 d^2 u^3 d u d u d^4) = f(u a d) = \sum\limits_{s \ge a} \V(a,s) I(s)$.

Clearly, there are five $s \ge a$ which have the same low valleys with $a$, namely: $s_i = u^2 d^2 a_i$, $i \in [5]$ where $a_1 = u^3 d u d u d^3$, $a_2 = u^4 d^2 u d^3$, $a_3 = u^3 d u^2 d^4$, $a_4 = u^4 d u d^4$ and $a_5 = u^5 d^5$.

We can easily check by the definition of $\V$ that $\V(a,s_1) = 1$, $\V(a,s_2) = \V(a,s_3) = 2$, $\V(a,s_4) = 4$ and $\V(a,s_5) = 5$.

Furthermore, using formula \eqref{eq:interval1} we have that 
\begin{align*}
I(s_1) & = |\left[ (2,2,5,6,7,7,7), (7,7,7,7,7,7,7) \right]| = |\left[ (2,2,5,6), (7,7,7,7) \right]| \\ 
& = \begin{vmatrix} 
6 & \binom{6}{2} & \binom{3}{3} & \binom{2}{4} \\[0.25em] 
1 & 6 & \binom{3}{2} & \binom{2}{3} \\[0.25em] 
0 & 1 & 3 & \binom{2}{2} \\[0.25em] 
0 & 0 & 1 & 2 
\end{vmatrix} = 71, \\
I(s_2) & = |\left[ (2,2,6,6,7,7,7), (7,7,7,7,7,7,7) \right]| = |\left[
(2,2,6,6), (7,7,7,7) \right]| \\ 
& = \begin{vmatrix} 6 & \binom{6}{2} & \binom{2}{3} & \binom{2}{4} \\[0.25em] 1 & 6 & \binom{2}{2} & \binom{2}{3} \\[0.25em] 0 & 1 & 2 & \binom{2}{2} \\[0.25em] 0 & 0 & 1 & 2 \end{vmatrix} = 51, \\
I(s_3) & = |\left[ (2,2,5,7,7,7,7), (7,7,7,7,7,7,7) \right]| = | \left[
(2,2,5), (7,7,7) \right]| = \begin{vmatrix} 6 & \binom{6}{2} & \binom{3}{3} \\[0.25em] 1 & 6 & \binom{3}{2} \\[0.25em] 0 & 1 & 3 \end{vmatrix} = 46, \\
I(s_4) & = |\left[ (2,2,6,7,7,7,7), (7,7,7,7,7,7,7) \right]| = | \left[
(2,2,6), (7,7,7) \right]| = \begin{vmatrix} 6 & \binom{6}{2} & \binom{2}{3} \\[0.25em] 1 & 6 & \binom{2}{2} \\[0.25em] 0 & 1 & 2 \end{vmatrix} = 36, \\
I(s_5) & = |\left[ (2,2,7,7,7,7,7), (7,7,7,7,7,7,7) \right]| = | \left[ (2,2), (7,7) \right] | = \begin{vmatrix} 6 & \binom{6}{2} \\[0.25em] 1 & 6 \end{vmatrix} = 21.
\end{align*}

From the above we obtain that $f(u^3 d^2 u^3 d u d u d^4) = 514$. \end{Example}

In the following result we show that, for specific Dyck paths, the map $f$ is related to the zeta function. We recall that the zeta function of a poset $X$ is given by the formula 
\[ \zeta(x, y) = \begin{cases} 1, & \textrm{$x \le y$}; \\ 0, & \textrm{ otherwise,}\end{cases}\] for $x, y \in X$. Moreover, it is well known (e.g., see \cite[p.\ 263]{Stanley2011}) that $\zeta^k(x,y)$ counts the number of $x-y$ multichains in $X$ of length $k$, for $k \ge 0$.

\begin{Corollary}
If $a$ is a product of pyramids, then 
\[ f(u^k a d^k) = \zeta^{k+1}(a, u^{|a|/2} d^{|a|/2}), \]
for every $k \ge 0$.
\end{Corollary}

\begin{proof}
For $k = 0$ the result obviously holds.

We now prove that for $k \ge 0$ we have that
\[ \V(u^k a d^k, u^k s d^k) = \zeta^k(a,s). \]
Indeed, we can easily see that $\hv(u^k a d^k) = k$, and that every $k$-multichain from $u^k a d^k$ to $u^k s d^k$ of type $\V$ is of the form \[ u^k a d^k = u^k \sigma_0 d^k \le \cdots \le u^k \sigma_{k-1} d^k \le u^k \sigma_k d^k = u^k s d^k, \]
producing the $a-s$ multichain 
\[ a = \sigma_0 \le \cdots \le \sigma_{k-1} \le \sigma_k = s. \]
Then, for $k \ge 1$, by Proposition \ref{prop:fprimedyck} we have that
\begin{align*}
f(u^k a d^k) 
& = \sum\limits_{s \ge u^{k-1} a d^{k-1}} \V(u^{k-1} a d^{k-1}, s) I(s) \\
& = \sum\limits_{s \ge a} \V(u^{k-1} a d^{k-1}, u^{k-1} s d^{k-1}) I(s) \\
& = \sum\limits_{s \ge a} \zeta^{k-1}(a,s) I(s) = \zeta^{k+1}(a, u^{|a|/2} d^{|a|/2}). \qedhere
\end{align*}
\end{proof}

Since $f$ is multiplicative on $\mathcal{D}$, by using formula \eqref{eq:fprimedyck} we can evaluate $f$ for every Dyck path. For Dyck prefixes that are not Dyck paths, it is enough to evaluate $f(duP)$, where $P$ is a Dyck prefix. We achieve this in formula \eqref{eq:fdyckprefix1} of the following Proposition, the proof of which, although it shares some common ideas with the proof of formula \eqref{eq:fprimedyck}, it is much more complicated. The difficulty lies in the fact that it is not possible to prove \eqref{eq:fdyckprefix1} directly by induction, so that we introduce and prove the more general equality \eqref{eq:fdyckprefix2} which concerns a pair of paths $(P,R)$ where a lexicographic induction applies. We also note that in this proof, we use several times the following obvious consequence of Proposition \ref{prop:basicrec}: 
\begin{equation}\label{eq:fdurec}
f(duP) = \sum\limits_{Q \in [P, \widetilde{P}]} f(d Q), \textrm{ for every Dyck prefix $P$}.
\end{equation}

\begin{Proposition}\label{prop:fdyckprefix}
For every Dyck prefix $P = a_0 u a_1 \cdots u a_k$, $k \ge 0$, $a_i \in \mathcal{D}$, $i \in [0,k]$, we have that
\begin{equation}\label{eq:fdyckprefix1}
f(d u P) = \sum \prod\limits_{i=0}^k \V(a_i, s_i) J(s_0 V_1),
\end{equation}
where the sum is taken over all sequences $(s_i)$, $i \in [0,k]$, of Dyck paths with $s_i \ge a_i$, and over all sequences of Dyck prefixes $(V_i)$, $i \in [k+1]$, with $V_i \ge u s_i V_{i+1}$, $i \in [k]$, and $V_{k+1} = \varepsilon$.
\end{Proposition}

\begin{proof} Let a Dyck prefix $R = b_0 u b_1 \cdots u b_\lambda$, $\lambda \ge 0$, $b_i \in \mathcal{D}$, $i \in [0,\lambda]$.

For the pair $(P, R)$ we consider the equality
\begin{equation}\label{eq:fdyckprefix2}
\sum\limits_{Q \in [R, \widetilde{R}]} f(du P Q) = \sum \prod\limits_{i=0}^k \V(a_i, s_i) \prod\limits_{i=0}^\lambda \V(b_i,t_i) J(s_0 V_1), 
\end{equation} 
where the sum is taken over all sequences $(s_i)$, $i \in [0,k]$, $(t_i)$, $i \in [0,\lambda]$, of Dyck paths with $s_i \ge a_i$ and $t_i \ge b_i$, and over all sequences of Dyck prefixes $(V_i)$, $i \in [k+\lambda +2]$,	 with $V_i \ge u s_i V_{i+1}$, $i \in [k]$, $V_{k+1} \ge t_0 V_{k+2}$, $V_{k+i+1} \ge u t_i V_{k+i+2}$, $i \in [\lambda]$, and $V_{k+\lambda+2} = \varepsilon$.

Clearly, formula \eqref{eq:fdyckprefix1} is a special case of equality \eqref{eq:fdyckprefix2} for $R = \varepsilon$. 

We prove equality \eqref{eq:fdyckprefix2} by induction. More precisely, assuming that \eqref{eq:fdyckprefix2} holds for every pair $(P_1, R_1)$ with either $|P_1| + |R_1| < |P| + |R|$, or $|P_1| + |R_1| = |P| + |R|$ with $|R_1| < |R|$, 
we prove that \eqref{eq:fdyckprefix2} also holds for the pair $(P,R)$.

We restrict ourselves to the case $P \notin \mathcal{D}$, since the case $P = a_0 \in \mathcal{D}$ is similar and easier to prove. We consider the following cases:

1.\ Assume that $R = u^\lambda$.

We first note that if $\lambda > 0$, then by the induction hypothesis we deduce that equality \eqref{eq:fdyckprefix2} holds for the pair $(P u^\lambda, \varepsilon)$, from which we can easily obtain that \eqref{eq:fdyckprefix2} holds also for the pair $(P, u^\lambda)$. Thus, we restrict ourselves to the case $R = \varepsilon$, so that now it is enough to prove formula \eqref{eq:fdyckprefix1}.

We set $H = a_1 u a_2 \cdots u a_k$ and we consider two subcases:

1(i).\ Assume that $a_0 = \varepsilon$, i.e., $P = u H$.
Then, by formula \eqref{eq:fdurec}, and by equality \eqref{eq:fdyckprefix2} for the pair $(\varepsilon, H)$, it follows that
\[ f(du P) = \sum\limits_{Q \in [H, \widetilde{H}]} f(du Q) = \sum \prod\limits_{i=1}^k \V(a_i, s_i) J(V_1), \]
where the sum it taken over all sequences $(s_i)$, $i \in [k]$, of Dyck paths with $s_i \ge a_i$ and over all sequences $(V_i)$, $i \in [k+1]$ with $V_1 \ge s_1 V_2$, $V_i \ge u s_i V_{i+1}$, $i \in [2,k]$, and $V_{k+1} = \varepsilon$.
Since $J(V_1) = J(u V_1)$, we can replace $V_1$ by $u V_1$, thus verifying formula \eqref{eq:fdyckprefix1}.

\smallskip

1(ii).\ Assume that $a_0 \ne \varepsilon$. We set $a_0 = u \gamma_1 d u \gamma_2 d \cdots u \gamma_\nu d$ and $\delta_r = u \gamma_{r+1} d u \gamma_{r+2} d \cdots u \gamma_\nu d$, $\nu \in \mathbb{N}^*$, $r \in [\nu]$.

Every $Q \in [P, \widetilde{P}]$ can be decomposed according to the leftmost common low valley with $P$ (see Figure \ref{fig:lclv}).

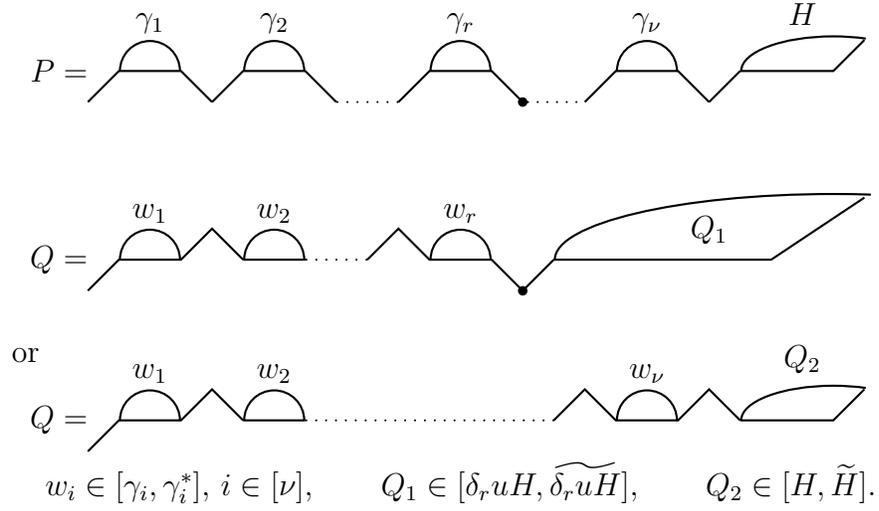
\begin{figure}[ht]
\psset{unit=1em}
\begin{center}
\begin{pspicture}(-2,0)(25,3)
\psline[](0,0)(1,1)(3,1)(4,0) \psellipticarc[](2,1)(1,1){0}{180} \rput[b](2,2.2){$\gamma_1$}
\psline[](4,0)(5,1)(7,1)(8,0) \psellipticarc[](6,1)(1,1){0}{180} \rput[b](6,2.2){$\gamma_2$}
\psline[linestyle=dotted](8,0)(10,0)
\psline[](10,0)(11,1)(13,1)(14,0) \psellipticarc[](12,1)(1,1){0}{180} \rput[b](12,2.2){$\gamma_r$}
\psline[linestyle=dotted]{*-}(14,0)(16,0)
\psline[](16,0)(17,1)(19,1)(20,0) \psellipticarc[](18,1)(1,1){0}{180} \rput[b](18,2.2){$\gamma_\nu$}
\psline[](20,0)(21,1)
\psellipticarc[](24,1)(3,1.15){45}{180} \psline[](21,1)(24,1)(25,2)
\rput[b](23,2.5){$H$}

\rput[r](0,1){$P=$}
\end{pspicture}

\vspace{2em}

\begin{pspicture}(-2,0)(25,4)
\rput[r](0,1){$Q=$}
\psline[](0,0)(1,1)(3,1)(4,2)(5,1) \psellipticarc[](2,1)(1,1){0}{180} \rput[b](2,2.2){$w_1$}
\psline[](5,1)(7,1) \psellipticarc[](6,1)(1,1){0}{180} \rput[b](6,2.2){$w_2$}
\psline[linestyle=dotted](7,1)(9,1)
\psline[](9,1)(10,2)(11,1)(13,1)(14,0)(15,1) \psellipticarc[](12,1)(1,1){0}{180} \rput[b](12,2.2){$w_r$} \psdots[](14,0)
\psellipticarc[](24,1)(9,2.15){60}{180} \psline[](15,1)(22,1)(25,3)
\rput[b](20,1.5){$Q_1$}
\end{pspicture}

\begin{pspicture}(-2,0)(25,2)
\rput[c](-2,0){or}
\end{pspicture}

\begin{pspicture}(-2,-1)(25,3)
\rput[r](0,1){$Q=$}
\psline[](0,0)(1,1)(3,1)(4,2)(5,1) \psellipticarc[](2,1)(1,1){0}{180} \rput[b](2,2.2){$w_1$}
\psline[](5,1)(7,1) \psellipticarc[](6,1)(1,1){0}{180} \rput[b](6,2.2){$w_2$}
\psline[linestyle=dotted](7,1)(15,1)
\psline[](15,1)(16,2)(17,1)(19,1)(20,2)(21,1) 
\psellipticarc[](18,1)(1,1){0}{180} \rput[b](18,2.2){$w_\nu$}
\psellipticarc[](24,1)(3,1.15){45}{180} \psline[](21,1)(24,1)(25,2)
\rput[b](23,2.5){$Q_2$}

\rput[c](12,-1){$w_i \in [\gamma_i, \gamma_i^*]$, $i\in [\nu], \qquad
Q_1 \in [\delta_r u H, \widetilde{\delta_r u H}], \qquad Q_2 \in [H, \widetilde{H}]$.}
\end{pspicture}
\end{center}
\caption{The decomposition of $Q \in [P, \widetilde{P}]$ according to the leftmost common low valley of $P,Q$.}
\label{fig:lclv}
\end{figure}

Using formula \eqref{eq:fdurec} and the above decomposition we have that
\begin{align}\label{eq:fdup1}
f(d u P) & = \sum\limits_{Q \in [P, \widetilde{P}]} f(dQ) \nonumber \\
& \begin{aligned} & = \sum\limits_{r=1}^{\nu} \sum\limits_{\stackrel{w_i \in [\gamma_i, \gamma_i^*]}{i \in [r]}} \sum\limits_{Q \in [\delta_r u H, \widetilde{\delta_r u H}]} f(d u w_1 u d \cdots w_{r-1} u d w_r d Q) \\ 
& + \sum\limits_{\stackrel{w_i \in [\gamma_i, \gamma_i^*]}{i \in [\nu]}} \sum\limits_{Q \in [H, \widetilde{H}]} f(d u w_1 u d \cdots w_{\nu} u d Q) 
\end{aligned} \nonumber \\
& \begin{aligned}
& = \sum\limits_{r=1}^{\nu} \sum\limits_{\stackrel{w_i \in [\gamma_i, \gamma_i^*]}{i \in [r]}} f(u w_1 u d \cdots w_{r-1} u d w_r d) \sum\limits_{Q \in [\delta_r u H, \widetilde{\delta_r u H}]} f(d Q)  \\ 
& + \sum\limits_{\stackrel{w_i \in [\gamma_i, \gamma_i^*]}{i \in [\nu]}} \sum\limits_{Q \in [H, \widetilde{H}]} f(d u w_1 u d \cdots w_{\nu} u d Q).
\end{aligned}
\end{align}

Then, using Proposition \ref{prop:fprimedyck}, and equality \eqref{eq:fdyckprefix2} for the pair $(w_1 u d \cdots w_{\nu} ud, H)$ we have that
\begin{align*}
f(duP) & = \sum\limits_{r=1}^{\nu} \sum\limits_{\stackrel{w_i \in [\gamma_i, \gamma_i^*]}{i \in [r]}} \sum\limits_{\stackrel{t_i \ge w_i}{i \in [r]}} \prod\limits_{i=1}^r \V(w_i, t_i) I(t_1 u d \cdots t_{r-1} u d t_r) f(d u \delta_r u H) \\ 
& \quad + \sum\limits_{w_i \in [\gamma_i, \gamma_i^*]} \sum\limits_{\stackrel{t_i \ge w_i}{i \in [\nu]}} \sum\limits_{\stackrel{s_i \ge a_i}{i \in [k]}} \sum\limits_{\substack{W_1 \ge s_1 V_2 \\ V_i \ge u s_i V_{i+1}, i \in [2,k] \\ V_{k+1} = \varepsilon}} \prod\limits_{i=1}^{\nu} \V(w_i, t_i) \prod\limits_{i=1}^k \V(a_i, s_i) J(t_1 u d \cdots t_\nu u d W_1) \nonumber \\ 
& = \sum\limits_{r=1}^{\nu} \sum\limits_{\stackrel{t_i \ge \gamma_i}{i \in [r]}} \left( \prod\limits_{i=1}^r \sum\limits_{\stackrel{w_i \in [\gamma_i, \gamma_i^*]}{i \in [r]}} \V(w_i, t_i) \right) I(t_1 u d \cdots t_{r-1} ud t_r) f(du \delta_r u H) \\
& \quad + \sum\limits_{\stackrel{t_i \ge \gamma_i, i \in [\nu]}{s_i \ge a_i, i \in [k]}} \sum\limits_{\substack{W_1 \ge s_1 V_2 \\ V_i \ge u s_i V_{i+1}, i \in [2,k] \\ V_{k+1} = \varepsilon}} \prod\limits_{i=1}^{k} \V(a_i, s_i) \left( \prod\limits_{i=1}^{\nu} \sum\limits_{\stackrel{w_i \in [\gamma_i, \gamma_i^*]}{i \in [\nu]}} \V(w_i, t_i) \right) J(t_1 u d \cdots t_\nu u d W_1). \end{align*}

Finally, using formula \eqref{eq:fdyckprefix1} for the path $\delta_r u H$ and Proposition \ref{prop:vbijection}, we obtain that
\begin{align}\label{eq:fdup2}
& f(duP)  \nonumber  \\
& \begin{aligned}
& = \sum\limits_{r=1}^{\nu} \sum\limits_{\stackrel{t_i \ge \gamma_i}{i \in [r]}} \prod\limits_{i=1}^r \V(u \gamma_i d, u t_i d) I(t_1 u d \cdots t_{r-1} ud t_r) \cdot \\ & \sum\limits_{\stackrel{t_i \ge \gamma_i, i \in [r+1,\nu]}{s_i \ge a_i, i \in [k]}} \sum\limits_{\stackrel{V_i \ge u s_i V_{i+1}, i \in [k]}{V_{k+1} = \varepsilon}} \V(\delta_r, u t_{r+1} d \cdots u t_{\nu} d) \prod\limits_{i=1}^k \V(a_i, s_i) J(u t_{r+1} d \cdots u t_{\nu} d V_1) \\
& + \sum\limits_{\stackrel{t_i \ge \gamma_i, i \in [\nu]}{s_i \ge a_i, i \in [k]}} \sum\limits_{\substack{W_1 \ge s_1 V_2 \\ V_i \ge u s_i V_{i+1}, i \in [2,k] \\ V_{k+1} = \varepsilon}} \prod\limits_{i=1}^k \V(a_i, s_i) \prod\limits_{i=1}^{\nu} \V(u \gamma_i d, u t_i d) J(t_1 ud \cdots t_{\nu} ud W_1) 
\end{aligned}  \nonumber \\
& \begin{aligned}
& = \sum\limits_{\substack{s_i \ge a_i, i \in [0,k] \\ s_0 = u t_1 d \cdots u t_{\nu} d}} \prod\limits_{i=0}^k \V(a_i, s_i)  \cdot \\ & \sum\limits_{\substack{W_1 \ge s_1 V_2 \\ V_i \ge u s_i V_{i+1}, \\ i \in [2,k] \\ V_{k+1} = \varepsilon}} \left( \sum\limits_{r=1}^{\nu} I(t_1 ud \cdots t_{r-1} ud t_r) J(u t_{r+1} d \cdots u t_{\nu} d u W_1) + J(t_1 ud \cdots t_{\nu} ud W_1) \right). 
\end{aligned}
\end{align}

Clearly, since the number $\sum\limits_{r=1}^{\nu} I(t_1 ud \cdots t_{r-1} ud t_r) J(u t_{r+1} d \cdots u t_{\nu} d u W_1)$ (resp., the number $J(t_1 ud \cdots t_{\nu} ud W_1)$) counts the set of all paths greater than or equal to $s_0 u W_1$ that have (resp., do not have) low valleys we obtain that
\[ \sum\limits_{r=1}^{\nu} I(t_1 ud \cdots t_{r-1} ud t_r) J(u t_{r+1} d \cdots u t_{\nu} d u W_1) + J(t_1 ud \cdots t_{\nu} ud W_1) = J(s_0 u W_1), \]
so that, by setting $V_1 = u W_1$ in formula \eqref{eq:fdup2}, we obtain  \eqref{eq:fdyckprefix1}.

2.\ Assume that $R \ne u^\lambda$. Then, we set $R = u^{\mu} b_{\mu} u b_{\mu+1} \cdots u b_{\lambda}$, where $\mu$ is the least integer such that $b_{\mu} \ne \varepsilon$.

It is enough to restrict ourselves to the case where $\mu = 0$ (i.e., $b_0 \ne \varepsilon)$, since the general case follows easily by applying equality \eqref{eq:fdyckprefix2} for the pair $(Pu^{\mu}, b_{\mu} u b_{\mu+1} \cdots u b_{\lambda})$.

Furthermore, we can restrict ourselves to the case where $R \notin \mathcal{D}$ (i.e., $\lambda > 0$), since the case where $R \in \mathcal{D}$ is similar and easier to prove. 

Set $b_0 = u \eta_1 d u \eta_2 d \cdots u \eta_{\xi} d$, $\theta_r = u \eta_{r+1} d u \eta_{r+2} d \cdots u \eta_{\xi} d$, $\xi \in \mathbb{N}^*$, $r \in [\xi]$, and $T = b_1 u b_2 \cdots u b_{\lambda}$.

By decomposing each $Q \in [R, \widetilde{R}]$ according to the leftmost low common valley with $R$, as before we have that
\begin{align}\label{eq:sumfduPQ}
\sum\limits_{Q \in [R, \widetilde{R}]} f(d u P Q) 
= & \sum\limits_{r=1}^{\xi} \sum\limits_{\stackrel{w_i \in [\eta_i, \eta_i^*]}{i \in [r]}} \sum\limits_{Q \in [\theta_r u T, \widetilde{\theta_r u T}]} f(d u P u w_1 u d \cdots w_{r-1} u d w_r d Q) \nonumber \\
& + \sum\limits_{\stackrel{w_i \in [\eta_i, \eta_i^*]}{i \in [\xi]}} \sum\limits_{Q \in [T, \widetilde{T}]} f(du P u w_1 u d \cdots w_{\xi} ud Q).
\end{align}

It is enough to prove that the first (resp., second) sum on the RHS of the above equality is equal to the part of the second sum in equality \eqref{eq:fdyckprefix2} which has (resp., does not have) low valleys.

Clearly, by applying \eqref{eq:fdyckprefix2} for the pair $(P u w_1 ud \cdots w_{r-1} ud w_r, \theta_r u T)$ for $r \in [\xi]$, we have that
\begin{align}\label{eq:sumfduPuwetc}
\begin{aligned}
& \sum\limits_{Q \in [\theta_r u T, \widetilde{\theta_r u T}]} \!\! f(du P u w_1 u d \cdots w_{r-1} u d w_r d Q) \\ & = \sum \prod\limits_{i=0}^k \V(a_i,s_i) \V(u w_1 ud \cdots w_{r-1} ud w_r d, u s d) \V(\theta_r, z_r) \prod\limits_{i=1}^{\lambda} \V(b_i, t_i) J(s_0 V_1),
\end{aligned}
\end{align}
where the sum is taken over all sequences $(s_i)$, $i \in [0,k]$, $(t_i)$, $i \in [\lambda]$, $s$, $z_r$ and $(V_1, \ldots, V_k$, $W_{k+1}, V_{k+2}, \ldots$, $V_{k+\lambda})$ with $s_i \ge a_i$, $t_i \ge b_i$, $s \ge w_1 u d w_2 ud \cdots w_{r-1} ud w_r$, $z_r \ge \theta_r$, $V_i \ge u s_i V_{i+1}$, $i \in [k-1]$, $V_k \ge u s_k u s d W_{k+1}$, $W_{k+1} \ge z_r V_{k+2}$, $V_{k+i+1} \ge u t_i V_{k+i+2}$, $i \in [\lambda]$, and $V_{k+\lambda+2} = \varepsilon$.

Clearly, we have
\[ \V(u w_1 u d \cdots w_{r-1} u d w_r d, u s d) = 
\sum\limits_{\phi \le s} \V(w_1 u d \cdots w_{r-1} ud w_r, \phi) = \sum\limits_{\stackrel{\phi_i \ge w_i}{i \in [r]}} \prod\limits_{i=1}^r \V(w_i, \phi_i),
\]
where $\phi_i \in \mathcal{D}$, $i \in [r]$, and $\phi_1 ud \cdots \phi_{r-1} ud \phi_r \le s$. Then, by substituting in equality \eqref{eq:sumfduPuwetc}, summing in terms of $(W_i)$, $i \in [r]$, changing the order of summation and using Proposition \ref{prop:vbijection}, in an analogous way as the one used in the proof of formula \eqref{eq:fdup1}, we can easily deduce that
\begin{align}\label{eq:doublesumfduPetc}
& \sum\limits_{\stackrel{w_i \in [\eta_i, \eta_i^*]}{i \in [r]}} \sum\limits_{Q \in [\theta_r u T, \widetilde{\theta_r u T}]} f(du P u w_1 ud \cdots w_{r-1} u d w_r d Q) \nonumber \\
& \qquad = \sum \prod\limits_{i=0}^k \V(a_i,s_i) \prod\limits_{i=1}^r \V(u \eta_i d, u \phi_i d) \V(\theta_r, z_r) \prod\limits_{i=1}^{\lambda} \V(b_i, t_i) J(s_0 V_1),
\end{align}
where $s_i \ge a_i$, $i \in [0,k]$, $t_i \ge b_i$, $i \in [\lambda]$, $\phi_i \ge \eta_i$, $i \in [r]$, $s \ge \phi_1 ud \cdots \phi_{r-1} u d \phi_r$, $z_r \ge \theta_r$, $V_i \ge u s_i V_{i+1}$, $i \in [k-1]$, $V_k \ge u s_k u s d W_{k+1}$, $W_{k+1} \ge z_r V_{k+2}$, $V_{k+i+1} \ge u t_i V_{k+i+2}$, $i \in [\lambda]$, and $V_{k + \lambda + 2} = \varepsilon$.

Then, by setting $t_0 = u \phi_1 d \cdots u \phi_r d z_r$ and $V_{k+1} = u s d W_{k+1}$, the RHS of equality \eqref{eq:doublesumfduPetc} becomes
\[ \sum \prod\limits_{i=0}^k \V(a_i, s_i) \prod\limits_{i=0}^{\lambda} \V(b_i, t_i) J(s_0 V_1), \]
where the sum is taken over all sequences $(s_i)$, $(t_i)$, $(V_i)$ with the same restrictions as in equality \eqref{eq:fdyckprefix2}, with the extra condition that the $r$-th low valley of $t_0$ (and $b_0$) coincides with the first low valley of $V_{k+1}$.

Next, by summing in terms of $r \in [\xi]$, we deduce that the first sum on the RHS of equality \eqref{eq:sumfduPQ} coincides with the part of the second sum in \eqref{eq:fdyckprefix2} for which $V_{k+1}$ has low valleys.

Finally, for the second sum on the RHS of equality \eqref{eq:sumfduPQ}, by applying \eqref{eq:fdyckprefix2} for the pair $(P u w_1 u d \cdots w_\xi ud, T)$, we have that
\begin{align*}
\sum\limits_{\stackrel{w_i \in [\eta_i, \eta_i^*]}{i \in [\xi]}} \sum\limits_{Q \in [T, \widetilde{T}]} f(du P u w_1 u d \cdots w_{\xi} ud Q) 
= \!\!\!\!\!\! \sum\limits_{\stackrel{w_i \in [\eta_i, \eta_i^*]}{i \in [\xi]}} \!\!\! \sum \prod\limits_{i=0}^k \V(a_i, s_i) \prod\limits_{i=1}^{\xi} \V(w_i, \phi_i) \prod\limits_{i=1}^\lambda \V(b_i, t_i) J(s_0 V_1), 
\end{align*}
where the sum is taken over all sequences $(s_i)$, $(\phi_i)$, $(b_i)$ with $s_i \ge a_i$, $i \in [0,k]$, $\phi_i \ge w_i$, $i \in [\xi]$, $b_i \ge t_i$, $i \in [\lambda]$, and over all sequences $(V_1,  \ldots, V_{k+1}, W_{k+2}, V_{k+3}, \ldots, V_{k+\lambda+2})$ with $V_i \ge u s_i V_{i+1}$, $i \in [k]$, $V_{k+1} \ge u \phi_1 ud \cdots \phi_{\xi} ud W_{k+2}$, $W_{k+2} \ge t_1 V_{k+3}$, $V_{k+i+1} \ge u t_i V_{k+i+2}$, $i \in [2,\lambda]$, and $V_{k+\lambda+2} = \varepsilon$.

By changing the order of summation on the RHS of the above equality, using Proposition \ref{prop:vbijection}, as before, and setting $t_0 = u \phi_1 d \cdots u \phi_\xi d$, $V_{k+2} = u W_{k+2}$ we deduce easily that the second sum on the RHS of equality \eqref{eq:sumfduPQ} coincides with the part of the second sum in \eqref{eq:fdyckprefix2} for which $V_{k+1}$ has no valleys, which completes the proof of equality \eqref{eq:fdyckprefix2}. \end{proof}

\begin{Example}
For $P = u^2 d u^2 d^3 u^2 d$, by formula \eqref{eq:fdyckprefix1} we have that $f(d u^3 d u^2 d^3 u^2 d) = f(d u P) = \sum \V(u^2 d u^2 d^3, s_0) \V(ud, s_1) J(s_0 V_1)$, where the sum is taken over all prime Dyck paths $s_0, s_1$ with $s_0 \ge u^2 d u^2 d^3$ and $s_1 \ge ud$, and all Dyck prefixes $V_1$ with $V_1 \ge u s_1$. 

It is easy to check that $s_0 = u^2 d u^2 d^3$ or $s_0 = u^4 d^4$, $s_1 = ud$ and $V_1 = u^2 d$ or $V_1 = u^3$.

Furthermore, for these $s_0, s_1$ we have that $\V(u^2 d u^2 d^3, s_0) = 1$ and $\V(ud, s_1) = 1$ so that
\begin{align*} 
f(d u^3 d u^2  d^3 u^2 d)  
& = J(u^2 d u^2 d^3 u^2 d) + J(u^2 d u^2 u^3 u^3) + J(u^3 d u d^3 u^2 d) + J(u^3 d u d^3 u^3)  \\ 
& \quad + J(u^4 d^4 u^2 d) + J(u^4 d^4 u^3).
\end{align*}

Then, by formulas \eqref{eq:interval2} and \eqref{eq:interval1} we obtain that
\begin{align*}
& J(u^2 d u^2 d^3 u^2 d) \\
& = |[u^{11}, u^{11}]| + |[u^2 d u^8, u^{10} d]| + |[u^2 d u^2 d u^5, u^9 d^2]| + |[u^2 d u^2 d^2 u^4, u^8 d^3]| \\ 
& \quad + |[u^2 d u^2 d^3 u^3, u^7 d^4]| + |[u^2 d u^2 d^3 u^2 d, u^6 d^5]| \\
& = 1 + |[(2),(10)]| + |[(2,4),(9,9)]| + |[(2,4,4),(8,8,8)]| \\ 
& \quad + |[(2,4,4,4),(7,7,7,7)]| + |[(2,4,4,4,6),(6,6,6,6,6)]| \\
& = 1 + 11 - 2 + \begin{vmatrix} 8 & \binom{6}{2} \\[0.25em] 1 & 6 \end{vmatrix} + \begin{vmatrix} 7 & \binom{5}{2} & \binom{5}{3} \\[0.25em] 1 & 5 & \binom{5}{2} \\[0.25em] 0 & 1 & 5 \end{vmatrix} + \begin{vmatrix} 6 & \binom{4}{2} & \binom{4}{3} & \binom{4}{4} \\[0.25em] 1 & 4 & \binom{4}{2} & \binom{4}{3} \\[0.25em] 0 & 1 & 4 & \binom{4}{2} \\[0.25em] 0 & 0 & 1 & 4 \end{vmatrix} + \begin{vmatrix} 5 & \binom{3}{2} & \binom{3}{3} & \binom{3}{4} & \binom{1}{5} \\[0.25em] 1 & 3 & \binom{3}{2} & \binom{3}{3} & \binom{1}{4} \\[0.25em] 0 & 1 & 3 & \binom{3}{2} & \binom{1}{3} \\[0.25em] 0 & 0 & 1 & 3 & \binom{1}{2} \\[0.25em] 0 & 0 & 0 & 1 & 1 \end{vmatrix} \\
& = 1 + 9 + 33 + 65 + 75 + 35 = 218.
\end{align*}

Similarly, we find that $J(u^2 d u^2 d^3 u^3) = 183$, $J(u^3 d u d^3 u^2 d) = 166$, $J(u^3 d u d^3 u^3) = 141$, $J(u^4 d^4 u^2 d) = 114$ and $J(u^4 d^4 u^3) = 99$, so that we obtain that $f(d u^3 d u^2 d^3 u^2 d) = 921$. \end{Example}

In the last result, we show that for specific paths, the map $f$ is also related to the zeta function.

\begin{Corollary}
If $a$ is a product of pyramids, then
\[ f(d u^k a) = \zeta^{k+1}(a, u^{|a|}), \]
for every $k \ge 1$.
\end{Corollary}

\begin{proof}
We apply Proposition~\ref{prop:fdyckprefix} for the Dyck prefix $P = u^{k-1} a$. Firstly, for $k = 1$ we obtain that
\[ f(du a) = J(a) = \zeta^2(a,u^{|a|}). \]
Next, for $k \ge 2$ we obtain that
\[ f(du a) = \sum J(V_1), \]
where the sum is taken over all sequences $(V_i)$, $i \in [k]$ of Dyck prefixes with $V_i \ge u V_{i+1}$, $i \in [k-1]$, $V_{k-1} \ge u a V_k$ and $V_k = \varepsilon$.

Each such sequence $(V_i)$, $i \in [k]$, produces a unique multichain $(W_i)$, $i \in [0,k-1]$, of length $k-1$ from $u^{k-1} a$ to $V_1$, defined by
$W_0 = u^{k-1}a$ and $W_i = u^{k-i-1} V_{k-i}$, $i \in [k-1]$.

Then, formula \eqref{eq:fdyckprefix1} gives that
\[ f(dua) = \sum\limits_{V_1 \ge u^{k-1}a} \zeta^{k-1}(u^{k-1}a, V_1) J(V_1) = \zeta^{k+1}(u^{k-1}a, u^{k-1+|a|}) = \zeta^{k+1}(a, u^{|a|}). \qedhere \]
\end{proof}

\bigskip
\hrule
\bigskip

\noindent 2010 {\it Mathematics Subject Classification}: Primary 05A19, Secondary 05A15, 06A07.

\noindent \emph{Keywords: } binary path, Dyck path, partial order of paths, chain of paths.

\bigskip
\hrule
\bigskip

\noindent
(Concerned with sequences
\seqnum{A000108},
\seqnum{A000213}, and
\seqnum{A001405}.)

\end{document}